\documentclass[notitlepage,12pt,reqno]{amsart}
\usepackage{amsmath, amssymb, amsfonts, amstext, amsthm, color,tocvsec2}
\usepackage[breaklinks=true]{hyperref}
\usepackage[mathscr]{euscript}
\usepackage{enumerate}
\usepackage[letterpaper]{geometry}
\usepackage{float}
\usepackage{pgf,tikz,ulem}
\newtheorem{thm}{Theorem}[section]
\newtheorem{defn}[thm]{Definition}
\newtheorem{lem}[thm]{Lemma}
\newtheorem{cor}[thm]{Corollary}
\newtheorem{prop}[thm]{Proposition}
\newtheorem{ex}[thm]{Example}
\newtheorem{rem}[thm]{Remark}

\def\be{\begin{equation}}
\def\ee{\end{equation}}
\def\bea{\begin{eqnarray}}
\def\eea{\end{eqnarray}}
\numberwithin{equation}{section}
\begin{document}
	
	\title[Matrices with simple symmetric digraphs and their group inverses]{Matrices with simple symmetric digraphs and their group inverses}

	\author{Raju Nandi}
	\address[R.~Nandi]{Department of Sciences, Indian Institute of Information Technology Design and Manufacturing, Kurnool, Andhra Pradesh 518008, India}
	\email{\tt rajunandirkm@iiitk.ac.in, rajunandirkm@gmail.com}
	
	\date{\today}
	
	\begin{abstract}
		A new class of simple symmetric digraphs called $\mathcal{D}$ is defined and studied here. Any digraph in $\mathcal{D}$ has the property that each non-pendant vertex is adjacent to at least one pendant vertex. A graph theoretical description for the entries of the group inverse of a real square matrix with any digraph belonging to this class is given. We classify all the real square matrices $A$ such that the digraphs associated with $A$ and $A^{\#}$ both are in $\mathcal{D}$, that is, the digraph related to $A$ is either a corona or a star digraph.
	\end{abstract}
	
	

 \maketitle
\section{Introduction}
Firstly, let us recall the definition of the group inverse of a matrix, the object of primary interest here. For a real $n\times n$ matrix $A$, the {\it group inverse}, if it exists, is the matrix $X$ that satisfies the equations $AXA=A, XAX=X$ and $AX=XA$. Such an $X$ is always unique and denoted by $A^{\#}$. It is well-known that the group inverse of $A$ exists if and only if $rank(A)=rank(A^2)$. For a non-singular matrix $A$, $A^{-1}=A^{\#}$. Let us recall that for a real rectangular matrix $A$, the {\it Moore-Penrose inverse} of $A$, is the unique matrix $A^{\dagger}$ that satisfies the equations $AA^{\dagger}A=A,~A^{\dagger}AA^{\dagger}=A^{\dagger},~(AA^{\dagger})^T=AA^{\dagger}$ and $(A^{\dagger}A)^T=A^{\dagger}A$. We refer the reader to \cite{bg} for more details on these notions of generalized inverses and Moore-Penrose inverses.

In combinatorial matrix theory, an interesting problem is giving the inverse or group inverse of a matrix using its graph structure. Here are some articles that explain how to determine the inverse \cite{akbari}, \cite{ghorbani}, \cite{sms}, \cite{car}, \cite{klein}, \cite{sj} and the group inverse \cite{bov}, \cite{encinas}, \cite{cov}, \cite{mdp}, \cite{maca}, \cite{nandi}, \cite{pavsi}. In chemistry, the group inverse of adjacency matrices is used to determine the separation gap of a molecule using graph energies \cite{sn}.
In addition, the group inverses of the matrices have applications in algebraic connectivity \cite{kirkland}, Markov chains \cite{campbell} and resistance distance \cite{bu}.

We recall some well-known results on the role of generalized inverses in graph theory. In \cite[Theorem 2.2]{cov}, authors computed the group inverse of the bipartite matrix of the form $A=\begin{pmatrix}
0 & B\\
C & 0
\end{pmatrix}$ and presented a graphical description for the entries of $A^{\#}$ for path digraph $D(A)$. In the same article, they proposed a conjecture (\cite[Conjecture 5.1]{cov}) on the entries of $A^{\#}$ when $D(A)$ is a tree digraph and recently, we solved this conjecture for a special class of trees in \cite[Theorem 2.11]{nandi}. A block-wise description for $A^{\#}$ was given in \cite[Theorem 2.1]{mdp} when $B,C$ of the form $[X\ U], [Y\ V]^T$ respectively, with $X,Y$ non-singular and $rank(UV)=1$. Also, a graphical description for the entries of $A^{\#}$ such that $D(A)$ is a broom tree was derived in \cite{mdp}. In \cite{pavsi}, the authors gave a graphical description for entries of the adjacency matrix of arbitrary weighted trees in terms of maximum matchings and alternating paths. Group inverses matrices associated with cycle graphs were investigated in \cite{mona,maca,sn,sivakumar}.

Now, we recall some graph-theoretic notations coined from \cite{cov}. Let $A=(a_{ij})\in \mathbb{R}^{n\times n}$. Then, the digraph corresponding to $A$, denoted by $D(A)=(V,E)$, is the directed graph whose vertex set is $V=\{1,2, \ldots ,n\}$ and the edge set $E$ is defined as follows: $(i,j)\in E$ iff $a_{ij}\neq 0$. A sequence $(i_1,i_2,\ldots ,i_m,i_{m+1})$ of $m+1$ distinct vertices with edges $(i_1,i_2),(i_2,i_3),\ldots ,(i_m,i_{m+1})$ is called a path of length $m$ from vertex $i_1$ to vertex $i_{m+1}$ in $D(A)$. When the first and last vertex of the above sequence are the same, i.e. $i_1=i_{m+1}$, it is called a $m$-$cycle$ (a cycle of length $m$) in $D(A)$. 

A digraph $D$ is called a {\it simple symmetric digraph} if it has no loops and there is an edge in each direction between two distinct vertices. Recall that a digraph that has a path from each vertex to every other vertex is called {\it strongly connected}. A digraph is said to be {\it tree} if it is a strongly connected digraph, and all of its cycles have length 2. Clearly, a tree digraph is a simple symmetric digraph. Suppose $A$ is a real square matrix such that $D(A)$ is a simple symmetric digraph. A 2-cycle $(i,j,i)$ is said to be incident to $i$ as well as $j$ in $D(A)$, and the vertex $i$ is called a pendant vertex if it is incident to only one 2-cycle. Otherwise, we will call $i$ a non-pendant vertex. On the other hand, a 2-cycle $(i,j,i)$ will be called a pendant cycle if at least one vertex $i$ or $j$ is pendant in $D(A)$, while a 2-cycle which is not pendant will be called a non-pendant cycle. A pair of vertices $i,j$ is said to be adjacent to each other if there is a 2-cycle $(i,j,i)$ in $D(A)$.

Now, extending the notations in \cite{nandi}, for a simple symmetric digraph $D(A)$, recall that for an even $r$, a set of $\frac{r}{2}$ disjoint 2-cycles in $D(A)$ given by $\{(i_1,i_2,i_1),(i_3,i_4,i_3),\ldots,(i_{r-1},i_r,i_{r-1})\},$  is called a {\it matching} and the product $a_{i_1,i_2}a_{i_2,i_1}a_{i_3,i_4}a_{i_4,i_3}\ldots a_{i_{r-1},i_r}a_{i_r,i_{r-1}}$ is called a {\it matching product}. If this set of 2-cycles has a maximum cardinality then the {\it matching} is referred to as a {\it maximum matching} and the matching product is called a {\it maximum matching product}. The sum of all maximum matching products in $D(A)$ is denoted by $\Delta_A$. A matching is said to be a {\it perfect matching} if it covers all the vertices of $D(A)$. Let $\mathbb{M}$ and $\mathbb{M}(i)$ denote the set of all maximum matchings in $D(A)$ and the set of all maximum matchings in which vertex $i$ is matched, respectively. For a cycle $(i,j,i)$ in $D(A)$, the product $a_{ij}a_{ji}$ is called the {\it cycle product}. A sequence of $m$ 2-cycles $((i_1,i_2,i_1),(i_2,i_3,i_2),\ldots,(i_m,i_{m+1},i_m))$ with $m+1$ distinct vertices $i_1,i_2,\ldots ,i_{m+1}$ in $D(A)$ is called a {\it cycle chain} from $i_1$ to $i_{m+1}$ of length $m$ and denoted by $C_m(i_1,i_{m+1})$. A cycle chain $C_m(i_1,i_{m+1})$ is said to be an alternating cycle chain with respect to a maximum matching $M$ if cycles of $C_m(i_1,i_{m+1})$ alternatively belong to $M$ and $M^c$, with the condition that both the first and the last cycle of $C_m(i_1,i_{m+1})$ belong to $M$. For an alternating cycle chain $C_m(i_1,i_{m+1})$, the product $a_{i_1,i_2}a_{i_2,i_3}a_{i_3,i_4} \ldots a_{i_{m-1},i_m}a_{i_m,i_{m+1}}$ is said to be the {\it path product along alternating cycle chain $C_m(i_1,i_{m+1})$}, denoted by $P_m(i_1,i_{m+1})$.

Let us introduce a new class of simple symmetric digraph.
\begin{defn}
Let $\mathcal{D}$ denote the set of all simple symmetric digraphs $D$  such that each non-pendant vertex of $D$ is adjacent to at least one pendant vertex of $D$.
\end{defn}
\begin{ex}
It is clear that the digraph $D_1 \in \mathcal{D},$ (Fig. \ref{D1}) while $D_2 \notin \mathcal{D}$ (Fig. \ref{D2}). The non-pendant vertex $4$ (in $D_2$) is not adjacent to any pendant vertex.
\begin{figure}[H]

\tikzset{every picture/.style={line width=0.75pt}} 

\begin{tikzpicture}[x=0.75pt,y=0.75pt,yscale=-1,xscale=1]

\draw   (257,118.5) .. controls (257,114.91) and (259.91,112) .. (263.5,112) .. controls (267.09,112) and (270,114.91) .. (270,118.5) .. controls (270,122.09) and (267.09,125) .. (263.5,125) .. controls (259.91,125) and (257,122.09) .. (257,118.5) -- cycle ;
\draw   (347,118.5) .. controls (347,114.91) and (349.91,112) .. (353.5,112) .. controls (357.09,112) and (360,114.91) .. (360,118.5) .. controls (360,122.09) and (357.09,125) .. (353.5,125) .. controls (349.91,125) and (347,122.09) .. (347,118.5) -- cycle ;
\draw   (256,199.5) .. controls (256,195.91) and (258.91,193) .. (262.5,193) .. controls (266.09,193) and (269,195.91) .. (269,199.5) .. controls (269,203.09) and (266.09,206) .. (262.5,206) .. controls (258.91,206) and (256,203.09) .. (256,199.5) -- cycle ;
\draw   (347,200.5) .. controls (347,196.91) and (349.91,194) .. (353.5,194) .. controls (357.09,194) and (360,196.91) .. (360,200.5) .. controls (360,204.09) and (357.09,207) .. (353.5,207) .. controls (349.91,207) and (347,204.09) .. (347,200.5) -- cycle ;
\draw    (270,118.5) .. controls (299.76,102.66) and (330.44,111.24) .. (345.23,117.71) ;
\draw [shift={(347,118.5)}, rotate = 204.9] [color={rgb, 255:red, 0; green, 0; blue, 0 }  ][line width=0.75]    (10.93,-3.29) .. controls (6.95,-1.4) and (3.31,-0.3) .. (0,0) .. controls (3.31,0.3) and (6.95,1.4) .. (10.93,3.29)   ;
\draw    (353.5,194) .. controls (365.63,166.84) and (363.64,144.38) .. (354.38,126.63) ;
\draw [shift={(353.5,125)}, rotate = 60.95] [color={rgb, 255:red, 0; green, 0; blue, 0 }  ][line width=0.75]    (10.93,-3.29) .. controls (6.95,-1.4) and (3.31,-0.3) .. (0,0) .. controls (3.31,0.3) and (6.95,1.4) .. (10.93,3.29)   ;
\draw    (263.5,194) .. controls (272.72,166.84) and (273.46,144.38) .. (264.37,126.63) ;
\draw [shift={(263.5,125)}, rotate = 60.95] [color={rgb, 255:red, 0; green, 0; blue, 0 }  ][line width=0.75]    (10.93,-3.29) .. controls (6.95,-1.4) and (3.31,-0.3) .. (0,0) .. controls (3.31,0.3) and (6.95,1.4) .. (10.93,3.29)   ;
\draw    (269,199.5) .. controls (296.99,181.17) and (328.69,190.77) .. (344.36,198.65) ;
\draw [shift={(346,199.5)}, rotate = 208.07] [color={rgb, 255:red, 0; green, 0; blue, 0 }  ][line width=0.75]    (10.93,-3.29) .. controls (6.95,-1.4) and (3.31,-0.3) .. (0,0) .. controls (3.31,0.3) and (6.95,1.4) .. (10.93,3.29)   ;
\draw    (263.5,125) .. controls (253.32,142.46) and (253.95,173.56) .. (261.76,191.39) ;
\draw [shift={(262.5,193)}, rotate = 244.09] [color={rgb, 255:red, 0; green, 0; blue, 0 }  ][line width=0.75]    (10.93,-3.29) .. controls (6.95,-1.4) and (3.31,-0.3) .. (0,0) .. controls (3.31,0.3) and (6.95,1.4) .. (10.93,3.29)   ;
\draw    (347,200.5) .. controls (315.15,212.08) and (287.97,206.9) .. (270.83,200.23) ;
\draw [shift={(269,199.5)}, rotate = 22.38] [color={rgb, 255:red, 0; green, 0; blue, 0 }  ][line width=0.75]    (10.93,-3.29) .. controls (6.95,-1.4) and (3.31,-0.3) .. (0,0) .. controls (3.31,0.3) and (6.95,1.4) .. (10.93,3.29)   ;
\draw   (255,274.5) .. controls (255,270.91) and (257.91,268) .. (261.5,268) .. controls (265.09,268) and (268,270.91) .. (268,274.5) .. controls (268,278.09) and (265.09,281) .. (261.5,281) .. controls (257.91,281) and (255,278.09) .. (255,274.5) -- cycle ;
\draw   (176,200.5) .. controls (176,196.91) and (178.91,194) .. (182.5,194) .. controls (186.09,194) and (189,196.91) .. (189,200.5) .. controls (189,204.09) and (186.09,207) .. (182.5,207) .. controls (178.91,207) and (176,204.09) .. (176,200.5) -- cycle ;
\draw    (257.5,113.5) .. controls (249.66,97.82) and (246.14,80.22) .. (253.53,53.17) ;
\draw [shift={(254,51.5)}, rotate = 105.95] [color={rgb, 255:red, 0; green, 0; blue, 0 }  ][line width=0.75]    (10.93,-3.29) .. controls (6.95,-1.4) and (3.31,-0.3) .. (0,0) .. controls (3.31,0.3) and (6.95,1.4) .. (10.93,3.29)   ;
\draw   (253,45.5) .. controls (253,41.91) and (255.91,39) .. (259.5,39) .. controls (263.09,39) and (266,41.91) .. (266,45.5) .. controls (266,49.09) and (263.09,52) .. (259.5,52) .. controls (255.91,52) and (253,49.09) .. (253,45.5) -- cycle ;
\draw    (261.5,268) .. controls (255.16,250.94) and (253.1,229.6) .. (261.81,207.69) ;
\draw [shift={(262.5,206)}, rotate = 112.89] [color={rgb, 255:red, 0; green, 0; blue, 0 }  ][line width=0.75]    (10.93,-3.29) .. controls (6.95,-1.4) and (3.31,-0.3) .. (0,0) .. controls (3.31,0.3) and (6.95,1.4) .. (10.93,3.29)   ;
\draw    (352.5,269) .. controls (346.16,251.94) and (344.1,230.6) .. (352.81,208.69) ;
\draw [shift={(353.5,207)}, rotate = 112.89] [color={rgb, 255:red, 0; green, 0; blue, 0 }  ][line width=0.75]    (10.93,-3.29) .. controls (6.95,-1.4) and (3.31,-0.3) .. (0,0) .. controls (3.31,0.3) and (6.95,1.4) .. (10.93,3.29)   ;
\draw   (346,275.5) .. controls (346,271.91) and (348.91,269) .. (352.5,269) .. controls (356.09,269) and (359,271.91) .. (359,275.5) .. controls (359,279.09) and (356.09,282) .. (352.5,282) .. controls (348.91,282) and (346,279.09) .. (346,275.5) -- cycle ;
\draw    (190,200.5) .. controls (218.95,186.02) and (238.59,191.11) .. (255.2,198.67) ;
\draw [shift={(257,199.5)}, rotate = 205.2] [color={rgb, 255:red, 0; green, 0; blue, 0 }  ][line width=0.75]    (10.93,-3.29) .. controls (6.95,-1.4) and (3.31,-0.3) .. (0,0) .. controls (3.31,0.3) and (6.95,1.4) .. (10.93,3.29)   ;
\draw    (261.5,53) .. controls (268.85,69.17) and (271.88,83.9) .. (264,110.36) ;
\draw [shift={(263.5,112)}, rotate = 287.18] [color={rgb, 255:red, 0; green, 0; blue, 0 }  ][line width=0.75]    (10.93,-3.29) .. controls (6.95,-1.4) and (3.31,-0.3) .. (0,0) .. controls (3.31,0.3) and (6.95,1.4) .. (10.93,3.29)   ;
\draw    (358.5,206) .. controls (365.85,222.17) and (366.48,241.22) .. (358.5,267.86) ;
\draw [shift={(358,269.5)}, rotate = 287.18] [color={rgb, 255:red, 0; green, 0; blue, 0 }  ][line width=0.75]    (10.93,-3.29) .. controls (6.95,-1.4) and (3.31,-0.3) .. (0,0) .. controls (3.31,0.3) and (6.95,1.4) .. (10.93,3.29)   ;
\draw    (267.5,205) .. controls (274.85,221.17) and (275.48,240.22) .. (267.5,266.86) ;
\draw [shift={(267,268.5)}, rotate = 287.18] [color={rgb, 255:red, 0; green, 0; blue, 0 }  ][line width=0.75]    (10.93,-3.29) .. controls (6.95,-1.4) and (3.31,-0.3) .. (0,0) .. controls (3.31,0.3) and (6.95,1.4) .. (10.93,3.29)   ;
\draw    (256,199.5) .. controls (239.43,207.3) and (213.34,209.4) .. (191.66,201.15) ;
\draw [shift={(190,200.5)}, rotate = 22.25] [color={rgb, 255:red, 0; green, 0; blue, 0 }  ][line width=0.75]    (10.93,-3.29) .. controls (6.95,-1.4) and (3.31,-0.3) .. (0,0) .. controls (3.31,0.3) and (6.95,1.4) .. (10.93,3.29)   ;
\draw    (355.5,110.51) .. controls (363.79,95.39) and (365.9,77.9) .. (357.19,53.86) ;
\draw [shift={(356.49,52)}, rotate = 69.18] [color={rgb, 255:red, 0; green, 0; blue, 0 }  ][line width=0.75]    (10.93,-3.29) .. controls (6.95,-1.4) and (3.31,-0.3) .. (0,0) .. controls (3.31,0.3) and (6.95,1.4) .. (10.93,3.29)   ;
\draw    (354,123.5) .. controls (344.2,147.02) and (344.96,165.26) .. (353,192.33) ;
\draw [shift={(353.5,194)}, rotate = 253.11] [color={rgb, 255:red, 0; green, 0; blue, 0 }  ][line width=0.75]    (10.93,-3.29) .. controls (6.95,-1.4) and (3.31,-0.3) .. (0,0) .. controls (3.31,0.3) and (6.95,1.4) .. (10.93,3.29)   ;
\draw    (346.5,118.5) .. controls (326.41,125.85) and (299.12,128.88) .. (271.68,119.11) ;
\draw [shift={(270,118.5)}, rotate = 20.56] [color={rgb, 255:red, 0; green, 0; blue, 0 }  ][line width=0.75]    (10.93,-3.29) .. controls (6.95,-1.4) and (3.31,-0.3) .. (0,0) .. controls (3.31,0.3) and (6.95,1.4) .. (10.93,3.29)   ;
\draw   (350,45.49) .. controls (350,41.9) and (352.92,39) .. (356.51,39) .. controls (360.1,39) and (363,41.92) .. (363,45.51) .. controls (363,49.1) and (360.08,52) .. (356.49,52) .. controls (352.9,52) and (350,49.08) .. (350,45.49) -- cycle ;
\draw    (356.49,52) .. controls (348.17,66.7) and (344.16,85.24) .. (352.95,110.45) ;
\draw [shift={(353.5,112)}, rotate = 249.93] [color={rgb, 255:red, 0; green, 0; blue, 0 }  ][line width=0.75]    (10.93,-3.29) .. controls (6.95,-1.4) and (3.31,-0.3) .. (0,0) .. controls (3.31,0.3) and (6.95,1.4) .. (10.93,3.29)   ;
\draw    (360,118.5) .. controls (388.65,105.61) and (406.36,113.24) .. (420.08,117.43) ;
\draw [shift={(422,118)}, rotate = 195.95] [color={rgb, 255:red, 0; green, 0; blue, 0 }  ][line width=0.75]    (10.93,-3.29) .. controls (6.95,-1.4) and (3.31,-0.3) .. (0,0) .. controls (3.31,0.3) and (6.95,1.4) .. (10.93,3.29)   ;
\draw   (422,118) .. controls (422,114.41) and (424.91,111.5) .. (428.5,111.5) .. controls (432.09,111.5) and (435,114.41) .. (435,118) .. controls (435,121.59) and (432.09,124.5) .. (428.5,124.5) .. controls (424.91,124.5) and (422,121.59) .. (422,118) -- cycle ;
\draw    (422,118) .. controls (412.25,124.83) and (381.59,130.7) .. (361.52,119.4) ;
\draw [shift={(360,118.5)}, rotate = 32.01] [color={rgb, 255:red, 0; green, 0; blue, 0 }  ][line width=0.75]    (10.93,-3.29) .. controls (6.95,-1.4) and (3.31,-0.3) .. (0,0) .. controls (3.31,0.3) and (6.95,1.4) .. (10.93,3.29)   ;

\draw (241,114.4) node [anchor=north west][inner sep=0.75pt]    {$2$};
\draw (367,130.9) node [anchor=north west][inner sep=0.75pt]    {$1$};
\draw (364,186.4) node [anchor=north west][inner sep=0.75pt]    {$4$};
\draw (243,174.4) node [anchor=north west][inner sep=0.75pt]    {$3$};
\draw (255,17.4) node [anchor=north west][inner sep=0.75pt]    {$7$};
\draw (342,286.4) node [anchor=north west][inner sep=0.75pt]    {$10$};
\draw (252,286.4) node [anchor=north west][inner sep=0.75pt]    {$9$};
\draw (158,194.4) node [anchor=north west][inner sep=0.75pt]    {$8$};
\draw (350,19.4) node [anchor=north west][inner sep=0.75pt]    {$5$};
\draw (441,109.4) node [anchor=north west][inner sep=0.75pt]    {$6$};

\end{tikzpicture}

\caption{$D_1$}
\label{D1}
\end{figure}
\begin{figure}[H]

\tikzset{every picture/.style={line width=0.75pt}} 

\begin{tikzpicture}[x=0.75pt,y=0.75pt,yscale=-1,xscale=1]

\draw   (257,118.5) .. controls (257,114.91) and (259.91,112) .. (263.5,112) .. controls (267.09,112) and (270,114.91) .. (270,118.5) .. controls (270,122.09) and (267.09,125) .. (263.5,125) .. controls (259.91,125) and (257,122.09) .. (257,118.5) -- cycle ;
\draw   (347,118.5) .. controls (347,114.91) and (349.91,112) .. (353.5,112) .. controls (357.09,112) and (360,114.91) .. (360,118.5) .. controls (360,122.09) and (357.09,125) .. (353.5,125) .. controls (349.91,125) and (347,122.09) .. (347,118.5) -- cycle ;
\draw   (256,199.5) .. controls (256,195.91) and (258.91,193) .. (262.5,193) .. controls (266.09,193) and (269,195.91) .. (269,199.5) .. controls (269,203.09) and (266.09,206) .. (262.5,206) .. controls (258.91,206) and (256,203.09) .. (256,199.5) -- cycle ;
\draw   (347,200.5) .. controls (347,196.91) and (349.91,194) .. (353.5,194) .. controls (357.09,194) and (360,196.91) .. (360,200.5) .. controls (360,204.09) and (357.09,207) .. (353.5,207) .. controls (349.91,207) and (347,204.09) .. (347,200.5) -- cycle ;
\draw    (270,118.5) .. controls (299.76,102.66) and (330.44,111.24) .. (345.23,117.71) ;
\draw [shift={(347,118.5)}, rotate = 204.9] [color={rgb, 255:red, 0; green, 0; blue, 0 }  ][line width=0.75]    (10.93,-3.29) .. controls (6.95,-1.4) and (3.31,-0.3) .. (0,0) .. controls (3.31,0.3) and (6.95,1.4) .. (10.93,3.29)   ;
\draw    (353.5,194) .. controls (365.63,166.84) and (363.64,144.38) .. (354.38,126.63) ;
\draw [shift={(353.5,125)}, rotate = 60.95] [color={rgb, 255:red, 0; green, 0; blue, 0 }  ][line width=0.75]    (10.93,-3.29) .. controls (6.95,-1.4) and (3.31,-0.3) .. (0,0) .. controls (3.31,0.3) and (6.95,1.4) .. (10.93,3.29)   ;
\draw    (263.5,194) .. controls (272.72,166.84) and (273.46,144.38) .. (264.37,126.63) ;
\draw [shift={(263.5,125)}, rotate = 60.95] [color={rgb, 255:red, 0; green, 0; blue, 0 }  ][line width=0.75]    (10.93,-3.29) .. controls (6.95,-1.4) and (3.31,-0.3) .. (0,0) .. controls (3.31,0.3) and (6.95,1.4) .. (10.93,3.29)   ;
\draw    (269,199.5) .. controls (296.99,181.17) and (328.69,190.77) .. (344.36,198.65) ;
\draw [shift={(346,199.5)}, rotate = 208.07] [color={rgb, 255:red, 0; green, 0; blue, 0 }  ][line width=0.75]    (10.93,-3.29) .. controls (6.95,-1.4) and (3.31,-0.3) .. (0,0) .. controls (3.31,0.3) and (6.95,1.4) .. (10.93,3.29)   ;
\draw    (263.5,125) .. controls (253.32,142.46) and (253.95,173.56) .. (261.76,191.39) ;
\draw [shift={(262.5,193)}, rotate = 244.09] [color={rgb, 255:red, 0; green, 0; blue, 0 }  ][line width=0.75]    (10.93,-3.29) .. controls (6.95,-1.4) and (3.31,-0.3) .. (0,0) .. controls (3.31,0.3) and (6.95,1.4) .. (10.93,3.29)   ;
\draw    (347,200.5) .. controls (315.15,212.08) and (287.97,206.9) .. (270.83,200.23) ;
\draw [shift={(269,199.5)}, rotate = 22.38] [color={rgb, 255:red, 0; green, 0; blue, 0 }  ][line width=0.75]    (10.93,-3.29) .. controls (6.95,-1.4) and (3.31,-0.3) .. (0,0) .. controls (3.31,0.3) and (6.95,1.4) .. (10.93,3.29)   ;
\draw   (255,274.5) .. controls (255,270.91) and (257.91,268) .. (261.5,268) .. controls (265.09,268) and (268,270.91) .. (268,274.5) .. controls (268,278.09) and (265.09,281) .. (261.5,281) .. controls (257.91,281) and (255,278.09) .. (255,274.5) -- cycle ;
\draw   (176,200.5) .. controls (176,196.91) and (178.91,194) .. (182.5,194) .. controls (186.09,194) and (189,196.91) .. (189,200.5) .. controls (189,204.09) and (186.09,207) .. (182.5,207) .. controls (178.91,207) and (176,204.09) .. (176,200.5) -- cycle ;
\draw    (257.5,113.5) .. controls (249.66,97.82) and (246.14,80.22) .. (253.53,53.17) ;
\draw [shift={(254,51.5)}, rotate = 105.95] [color={rgb, 255:red, 0; green, 0; blue, 0 }  ][line width=0.75]    (10.93,-3.29) .. controls (6.95,-1.4) and (3.31,-0.3) .. (0,0) .. controls (3.31,0.3) and (6.95,1.4) .. (10.93,3.29)   ;
\draw   (253,45.5) .. controls (253,41.91) and (255.91,39) .. (259.5,39) .. controls (263.09,39) and (266,41.91) .. (266,45.5) .. controls (266,49.09) and (263.09,52) .. (259.5,52) .. controls (255.91,52) and (253,49.09) .. (253,45.5) -- cycle ;
\draw    (261.5,268) .. controls (255.16,250.94) and (253.1,229.6) .. (261.81,207.69) ;
\draw [shift={(262.5,206)}, rotate = 112.89] [color={rgb, 255:red, 0; green, 0; blue, 0 }  ][line width=0.75]    (10.93,-3.29) .. controls (6.95,-1.4) and (3.31,-0.3) .. (0,0) .. controls (3.31,0.3) and (6.95,1.4) .. (10.93,3.29)   ;
\draw    (190,200.5) .. controls (218.95,186.02) and (238.59,191.11) .. (255.2,198.67) ;
\draw [shift={(257,199.5)}, rotate = 205.2] [color={rgb, 255:red, 0; green, 0; blue, 0 }  ][line width=0.75]    (10.93,-3.29) .. controls (6.95,-1.4) and (3.31,-0.3) .. (0,0) .. controls (3.31,0.3) and (6.95,1.4) .. (10.93,3.29)   ;
\draw    (261.5,53) .. controls (268.85,69.17) and (271.88,83.9) .. (264,110.36) ;
\draw [shift={(263.5,112)}, rotate = 287.18] [color={rgb, 255:red, 0; green, 0; blue, 0 }  ][line width=0.75]    (10.93,-3.29) .. controls (6.95,-1.4) and (3.31,-0.3) .. (0,0) .. controls (3.31,0.3) and (6.95,1.4) .. (10.93,3.29)   ;
\draw    (267.5,205) .. controls (274.85,221.17) and (275.48,240.22) .. (267.5,266.86) ;
\draw [shift={(267,268.5)}, rotate = 287.18] [color={rgb, 255:red, 0; green, 0; blue, 0 }  ][line width=0.75]    (10.93,-3.29) .. controls (6.95,-1.4) and (3.31,-0.3) .. (0,0) .. controls (3.31,0.3) and (6.95,1.4) .. (10.93,3.29)   ;
\draw    (256,199.5) .. controls (239.43,207.3) and (213.34,209.4) .. (191.66,201.15) ;
\draw [shift={(190,200.5)}, rotate = 22.25] [color={rgb, 255:red, 0; green, 0; blue, 0 }  ][line width=0.75]    (10.93,-3.29) .. controls (6.95,-1.4) and (3.31,-0.3) .. (0,0) .. controls (3.31,0.3) and (6.95,1.4) .. (10.93,3.29)   ;
\draw    (355.5,110.51) .. controls (363.79,95.39) and (365.9,77.9) .. (357.19,53.86) ;
\draw [shift={(356.49,52)}, rotate = 69.18] [color={rgb, 255:red, 0; green, 0; blue, 0 }  ][line width=0.75]    (10.93,-3.29) .. controls (6.95,-1.4) and (3.31,-0.3) .. (0,0) .. controls (3.31,0.3) and (6.95,1.4) .. (10.93,3.29)   ;
\draw    (354,123.5) .. controls (344.2,147.02) and (344.96,165.26) .. (353,192.33) ;
\draw [shift={(353.5,194)}, rotate = 253.11] [color={rgb, 255:red, 0; green, 0; blue, 0 }  ][line width=0.75]    (10.93,-3.29) .. controls (6.95,-1.4) and (3.31,-0.3) .. (0,0) .. controls (3.31,0.3) and (6.95,1.4) .. (10.93,3.29)   ;
\draw    (346.5,118.5) .. controls (326.41,125.85) and (299.12,128.88) .. (271.68,119.11) ;
\draw [shift={(270,118.5)}, rotate = 20.56] [color={rgb, 255:red, 0; green, 0; blue, 0 }  ][line width=0.75]    (10.93,-3.29) .. controls (6.95,-1.4) and (3.31,-0.3) .. (0,0) .. controls (3.31,0.3) and (6.95,1.4) .. (10.93,3.29)   ;
\draw   (350,45.49) .. controls (350,41.9) and (352.92,39) .. (356.51,39) .. controls (360.1,39) and (363,41.92) .. (363,45.51) .. controls (363,49.1) and (360.08,52) .. (356.49,52) .. controls (352.9,52) and (350,49.08) .. (350,45.49) -- cycle ;
\draw    (356.49,52) .. controls (348.17,66.7) and (344.16,85.24) .. (352.95,110.45) ;
\draw [shift={(353.5,112)}, rotate = 249.93] [color={rgb, 255:red, 0; green, 0; blue, 0 }  ][line width=0.75]    (10.93,-3.29) .. controls (6.95,-1.4) and (3.31,-0.3) .. (0,0) .. controls (3.31,0.3) and (6.95,1.4) .. (10.93,3.29)   ;
\draw    (360,118.5) .. controls (388.65,105.61) and (406.36,113.24) .. (420.08,117.43) ;
\draw [shift={(422,118)}, rotate = 195.95] [color={rgb, 255:red, 0; green, 0; blue, 0 }  ][line width=0.75]    (10.93,-3.29) .. controls (6.95,-1.4) and (3.31,-0.3) .. (0,0) .. controls (3.31,0.3) and (6.95,1.4) .. (10.93,3.29)   ;
\draw   (422,118) .. controls (422,114.41) and (424.91,111.5) .. (428.5,111.5) .. controls (432.09,111.5) and (435,114.41) .. (435,118) .. controls (435,121.59) and (432.09,124.5) .. (428.5,124.5) .. controls (424.91,124.5) and (422,121.59) .. (422,118) -- cycle ;
\draw    (422,118) .. controls (412.25,124.83) and (381.59,130.7) .. (361.52,119.4) ;
\draw [shift={(360,118.5)}, rotate = 32.01] [color={rgb, 255:red, 0; green, 0; blue, 0 }  ][line width=0.75]    (10.93,-3.29) .. controls (6.95,-1.4) and (3.31,-0.3) .. (0,0) .. controls (3.31,0.3) and (6.95,1.4) .. (10.93,3.29)   ;

\draw (241,114.4) node [anchor=north west][inner sep=0.75pt]    {$2$};
\draw (367,130.9) node [anchor=north west][inner sep=0.75pt]    {$1$};
\draw (364,186.4) node [anchor=north west][inner sep=0.75pt]    {$4$};
\draw (243,174.4) node [anchor=north west][inner sep=0.75pt]    {$3$};
\draw (255,17.4) node [anchor=north west][inner sep=0.75pt]    {$7$};
\draw (252,286.4) node [anchor=north west][inner sep=0.75pt]    {$9$};
\draw (158,194.4) node [anchor=north west][inner sep=0.75pt]    {$8$};
\draw (350,19.4) node [anchor=north west][inner sep=0.75pt]    {$5$};
\draw (441,109.4) node [anchor=north west][inner sep=0.75pt]    {$6$};

\end{tikzpicture}
\caption{$D_2$}
\label{D2}
\end{figure}
\end{ex}
In another way, the property of a digraph $D$ in $\mathcal{D}$ can be written as follows: If $P\subset V(D)$ is the set of pendant vertices, then $\delta (P)=V(D)\backslash P$ or in other words, $P$ is a dense subset of $V(D)$. Sometimes, this property is referred to as the set $P$ dominates $V(D)$. Remember that for a set $F\subset V(D)$ we have $\overline{F}=F\cup \delta (F)$, where $\delta (F)=\{x\in V(D)\backslash F: x \mbox{ is adjacent to } y \mbox{ for some } y\in F\}$.

Let us define more terminology for a strongly connected digraph $D(A)\in \mathcal{D}$, where $A$ is a real square matrix. In particular, when $D(A)$ is a tree digraph, these terminologies are defined in \cite{nandi}. Now, we are extending those terminologies for a simple symmetric digraph $D(A)\in \mathcal{D}$. For arbitrary vertices $i$ and $j$ in $D(A)$ denote $\mathbb{M}(i,j)$ to be the set of all maximum matchings $M$ in $D(A)$ such that $C_m(i,j)$ is an alternating cycle chain with respect to $M$. Clearly, $\mathbb{M}(i,j)=\mathbb{M}(j,i)$. A necessary condition for the set $\mathbb{M}(i,j)$ to be non-empty is that the length of the path from $i$ to $j$ be odd. If $(i,j,i)$ is a 2-cycle of some maximum matching, then $\mathbb{M}(i,j)$ is non-empty. Two distinct vertices $i$ and $j$ will be called {\it maximally matchable} if $\mathbb{M}(i,j)\neq \phi$ and by Proposition \ref{uniquealternating}, for a maximally matchable vertices $i,j$, $C_m(i,j)$ is unique.

Further, for any maximally matchable vertices $i,j$ and a maximum matching $M\in \mathbb{M}(i,j)$ let $\beta _{\overline{i,j}}(M)$ denote the product of all cycle product, ranging over all the cycles of $M$ that is not contained in the unique cycle chain $C_m(i,j)$ in $D(A)$ (product over an empty set is considered to be equal to 1). Since $\mathbb{M}(i,j)=\mathbb{M}(j,i)$, note that $\beta _{\overline{i,j}}(M)=\beta _{\overline{j,i}}(M)$. For a maximum matching $M$ in $D(A)$, $\eta (M)$ denotes the maximum matching product. Set 
\begin{center}
$\beta _{ij}=\left\{\begin{array}{cl} (-1)^{\frac{m-1}{2}}P_m(i,j)&\mbox{if }i,j\mbox{ are maximally matchable, }\\0&\mbox{if }i,j\mbox{ are not maximally matchable. }\end{array}\right.$    
\end{center}
and
\begin{equation}\label{mu}
\mu _{ij}=\beta _{ij} \cdot \sum _{M\in \mathbb{M}(i,j)}\beta _{\overline{i,j}}(M).   
\end{equation}
It follows $\mu _{ij}=0$ if $i,j$ are not maximally matchable. This includes the case $i=j$.

In this terminology, a formula for the entries of the group inverse of a matrix $A$ with tree digraph $D(A)\in \mathcal{D}$, proved in \cite{nandi}, is recalled next.
\begin{thm}\cite[Theorem 1.4]{nandi}\label{treeresult}
Let $A$ be an $n\times n$ real matrix with a tree digraph $D(A)\in \mathcal{D}$ and assume that $\Delta _A \neq 0$. Let $A^{\#}=(\alpha _{ij})$ and let $\mu _{ij}$ be defined as above. Then, $\alpha _{ij}=\frac{\mu _{ij}}{\Delta _A}$.
\end{thm}
In \cite{cov}, authors presented a graphical description for the entries of group inverse of a matrix $A$ with path digraph $D(A)$ and proposed a conjecture for tree digraph. Recently, in Theorem \ref{treeresult}, we gave a formula for the entries of the group inverse of a matrix with a special class of tree digraphs and proved that the conjecture proposed in \cite{cov} holds for this special class of tree digraphs. Our first main result shows that \cite[Theorem 1.4]{nandi} is extendable for the larger class $\mathcal{D}$.
\begin{thm}\label{finalresult}
Let $A=(a_{ij})$ be an $n\times n$ real matrix such that $D(A)\in \mathcal{D}$. Assume that $D(A)$ is a strongly connected digraph and $\Delta _A \neq 0$. Let $A^{\#}=(\alpha _{ij})$ and let $\mu _{ij}$ be defined in \ref{mu}. Then, $\alpha _{ij}=\frac{\mu _{ij}}{\Delta _A}$.
\end{thm}
Example \ref{ssd} shows that the group inverse of a real square matrix $A$ with a simple symmetric digraph may not be a matrix with a simple symmetric digraph, whereas $D(A^{\#})$ is a simple symmetric digraph if $D(A)\in \mathcal{D}$. Now, a natural question arises: When $D(A)$ and $D(A^{\#})$ both are in $\mathcal{D}$ for a real square matrix $A$? We give the answer in our final main result.
\begin{thm}\label{maintheorem}
Let $A$ be an $n\times n$ real matrix such that $D(A)\in \mathcal{D}$. Assume that $D(A)$ is strongly connected and $\Delta _A\neq 0$. Then $D(A^{\#})\in \mathcal{D}$ if and only if $D(A)$ is either a corona digraph or a star tree digraph.
\end{thm}
\noindent\textbf{Organization of the paper:} The remaining sections are devoted to proving our main results above. In section \ref{section2}, we give a few properties of digraphs in $\mathcal{D}$ and recall a block-wise description for the group inverse of a matrix $A$ with $D(A)\in \mathcal{D}$, and prove Theorem \ref{finalresult} by giving three Lemmas. In section \ref{section3}, we provide a necessary and sufficient condition for the $ij$-th entry of the matrix $A^{\#}$ to be non-zero for a real square matrix $A$ with $D(A)\in \mathcal{D}$ and prove Theorem \ref{maintheorem}.

\section{Graphical description for the group inverses of matrices with simple symmetric digraphs}\label{section2}
Recall that a real square matrix $A=(a_{ij})$ is called {\it combinatorially symmetric} if $a_{ij}\neq 0$ iff $a_{ji}\neq 0$. Trivially, any symmetric matrix is combinatorially symmetric. Parter was the first author to use the concept of combinatorially symmetric matrices \cite{parter}. In 1974, Maybee first introduced the notion of a combinatorially symmetric matrix in \cite{maybee}. Note that for a simple symmetric digraph $D(A)$, the corresponding matrix $A$ is a combinatorially symmetric matrix with zero diagonal.

An important point to note is that the underlying graphs (of this class $\mathcal{D}$) are a special case of cluster networks derived by taking arbitrary graphs as bases and stars as satellites, \cite{ca,zhang}. Cluster networks are highly relevant in applications in Chemistry since all composite molecules consisting of some amalgamation over a central submolecule can be understood as generalized cluster networks. For instance, they can be used to understand some issues in metal-metal interaction in some molecules, since a cluster network structure can be easily found. In \cite{zhang}, the Kirchhoff index formulae for composite graphs known as join, corona and cluster of two graphs, are presented, in terms of the pieces. The Kirchhoff index formulae and the effective resistances of generalized composite networks, such as generalized cluster or corona network are obtained, in terms of the pieces, in \cite{ca}. We refer the reader to \cite{ac,bendito} for more details on the kirchhoff index of networks.

A real square matrix $A$ is said to be {\it irreducible} if $D(A)$ is strongly connected. Let $N(i)$ denote the set of all vertices adjacent to vertex $i$. Suppose a digraph $D$ is in $\mathcal{D}$. Then, each component of $D$ also belongs to $\mathcal{D}$. Henceforth, we will consider a strongly connected digraph in $\mathcal{D}$. Now, we obtain some special properties for the digraphs in $\mathcal{D}$ in the following few results. In general, they are not valid for any simple symmetric digraph.

\begin{prop}\label{nononpendant}
Let $D\in \mathcal{D}$ be a strongly connected digraph. Then, no non-pendant cycle can belong to a maximum matching of D.
\end{prop}
\begin{proof}
In \cite[Proposition 2.2]{nandi}, we prove the same result for a tree digraph $D\in \mathcal{D}$. For any digraph $D\in \mathcal{D}$, we can obtain this result by giving the same proof.
\end{proof}
\begin{rem}\label{properties}
Let $D\in \mathcal{D}$ have $k$ non-pendant vertices. Then, the number of 2-cycles in a maximum matching is always $k$. Note that, every non-pendant vertex is matched in any maximum matching of $D$, and both the endpoints of a length three alternating cycle chain are pendant vertices and a length one alternating cycle chain is nothing but a pendant cycle.
\end{rem}
\begin{prop}\label{uniquealternating}
Let $D\in \mathcal{D}$ be a strongly connected digraph. We then have the following:\\
$(i)$ The length of any alternating cycle chain is at most three.\\
$(ii)$ Any alternating cycle chain between two vertices is unique.
\end{prop}
\begin{proof}
$(i)$ Suppose $D$ has an alternating cycle chain $C$ of length at least five. Then, $C$ must have at least one non-pendant maximum matching cycle, a contradiction to Proposition \ref{nononpendant}.

$(ii)$ Let $C$ be an alternating cycle chain from vertex $i$ to vertex $j$ in $D$. Then, by first part, the length of $C$ is at most three. If the length is one then $C$ is simply a pendant cycle. So, it is unique. Now, let us consider the case when the length of $C$ is three. Since pendant cycles are the only maximum matching cycles (corresponding to any maximum matching), the initial and the terminal vertices of $C$ should be pendant vertices. So, $i$ and $j$ are pendant vertices. Let $\tilde{C}$ be another alternating cycle chain from $i$ to $j$. Since $i$ and $j$ are pendant vertices, the initial and the terminal 2-cycles of $C$ and $\tilde{C}$ must be the same. Here, $\tilde{C}$ is an alternating cycle chain, so its length should be odd. $C$ and $\tilde{C}$ are two different alternating cycle chains with the condition that they have the same initial and terminal 2-cycles. This means that the length of the cycle chain $\tilde{C}$ should be at least five, a contradiction to the fact that the length of any alternating cycle chain in $D$ is at most three (as we proved in part (i)). Therefore, $C$ and $\tilde{C}$ coincide.
\end{proof}
\begin{rem}
For a tree digraph, there is a unique cycle chain between two arbitrary vertices. A simple symmetric digraph $D(A)\in \mathcal{D}$ may not have a unique cycle chain between two arbitrary vertices. If a cycle chain between two vertices is an alternating cycle chain, it is always a unique alternating cycle chain between them in $D(A)$.
\end{rem}
Let $A$ be an $n\times n$ real matrix such that $D(A)\in \mathcal{D}$. Then recall a blockwise description for $A^{\#}$ from \cite{mcdonald}. Suppose $D(A)$ has $k$ non-pendant vertices. For $i=1,2,\ldots ,k$, let $x_i$ and $y_i$ be (column) vectors of length $r_i \in \mathbb{N}$ such that every coordinate is nonzero $(\sum _{i=1}^kr_i=n-k)$. Set
\begin{center}
 $F=\begin{pmatrix}
x_1^T & 0 & \hdots & 0 \\
0 & x_2^T & \ddots & \vdots \\
\vdots & \ddots & \ddots & 0 \\
0 & \hdots & 0 & x_k^T \\
\end{pmatrix}\in \mathbb{R}^{k\times (n-k)}$ and $G=\begin{pmatrix}
y_1 & 0 & \hdots & 0 \\
0 & y_2 & \ddots & \vdots \\
\vdots & \ddots & \ddots & 0 \\
0 & \hdots & 0 & y_k \\
\end{pmatrix}\in \mathbb{R}^{(n-k)\times k}$.
\end{center}
Then $A$ can be written permutationally similar to a matrix of the form \begin{equation}\label{glsform}
\begin{pmatrix}
E & F  \\
G & 0 \\
\end{pmatrix},
\end{equation}
where $E=(e_{ij})\in \mathbb{R}^{k\times k}$. Set $\alpha _i=x_i^Ty_i$ for all $i\in \{1,2,\ldots ,k\}$. Let $W=(W_{ij})$ be the $k \times k$ block matrix, where $W_{ij}=(\frac{e_{ij}}{\alpha _i\alpha _j})y_ix_j^T$, an $r_i\times r_j$ matrix. Then, by \cite[Theorem 2.5]{mcdonald}, the group inverse of matrices of the form \ref{glsform} has the $2\times 2$ block form
\begin{equation}\label{gigls}
\begin{pmatrix}
0 & ~Y  \\
Z & -W \\
\end{pmatrix},
\end{equation}
where $Y=\frac{x_1^T}{\alpha _1}\oplus \frac{x_2^T}{\alpha _2}\oplus \ldots \oplus \frac{x_k^T}{\alpha _k}$ and $Z=\frac{y_1}{\alpha _1}\oplus \frac{y_2}{\alpha _2}\oplus \ldots \oplus \frac{y_k}{\alpha _k}$.

For a square matrix $A$, $A^{\#}$ may not always exist. For instance, if $A=\begin{pmatrix}
0 & 1  \\
0 & 0 \\
\end{pmatrix},$ then $A^{\#}$ does not exist. In particular, it is easy to show that the group inverse does not exist for any nilpotent matrix. The following proposition provides a necessary and sufficient condition for the existence of the matrix $A^{\#}$ when $D(A)\in \mathcal{D}$, that is the sum of all maximum matchings in $D(A)$ is nonzero.
\begin{prop}
Let $A$ be an $n\times n$ real matrix such that $D(A)\in \mathcal{D}$. Assume that $D(A)$ is strongly connected. Then $A^{\#}$ exists if  and only if $\Delta _A\neq 0$.
\end{prop}
\begin{proof}
$A$ has the form \ref{glsform} since $D(A)\in \mathcal{D}$. Now, set $\alpha _i=x_i^Ty_i$ for all $i\in \{1,2,\ldots ,k\}$. For each $i$, the non-pendant vertex $i$ is adjacent to $r_i$ pendant vertices and $\alpha _i$ is the sum of cycle products of $r_i$ pendant cycles incident $i$. Note that
\begin{center}
$FG=\begin{pmatrix}
\alpha _1 & 0 & \hdots & 0 \\
0 &\alpha _2 & \ddots & \vdots \\
\vdots & \ddots & \ddots & 0 \\
0 & \hdots & 0 & \alpha _k \\
\end{pmatrix}\in \mathbb{R}^{k\times k}$ and $F^{\dagger}=\begin{pmatrix}
\frac{x_1}{x_1^Tx_1} & 0 & \hdots & 0 \\
0 & \frac{x_2}{x_2^Tx_2} & \ddots & \vdots \\
\vdots & \ddots & \ddots & 0 \\
0 & \hdots & 0 & \frac{x_k}{x_k^Tx_k} \\
\end{pmatrix}$.
\end{center}
By \cite[Theorem 3.1]{lw}, $A^{\#}$ exists if and only if $rank(F)=rank(FG)$. Here, $rank(F)=k$ which implies $A^{\#}$ exists if and only if $\alpha _i\neq 0$ for all $i\in \{1,2,\ldots ,k\}$. By Proposition \ref{nononpendant}, all the maximum matching cycles are pendant cycles and a maximum matching of $D(A)$ has a set of $k$ pendant cycles incident to $k$ non-pendant vertices (the total number of maximum matching is $\prod _{i=1}^kr _i$). So, $\Delta _A=\prod _{i=1}^k\alpha _i$ and it follows that $A^{\#}$ exists if and only if $\Delta _A\neq 0$.
\end{proof}

\begin{rem}\label{twoproperties}
Let $A=(a_{ij})$ be an $n\times n$ real matrix such that $D(A)\in \mathcal{D}$. Let $\{i_1,i_2,\ldots ,i_{s+1}\}$ be a set of all pendant vertices with a common neighbour $q$. Then,\\
\rm{$(i)$} $\mathbb{M}=\cup _{m=1}^{s+1}\mathbb{M}(i_m)$ and $\cap _{m=1}^{s+1}\mathbb{M}(i_m)=\phi $\\
\rm{$(ii)$} $\sum _{M\in \mathbb{M}(i_m,q)}\beta _{\overline{i_m,q}}(M)=\sum _{M\in \mathbb{M}(i_t,q)}\beta _{\overline{i_t,q}}(M)$ for all $m,t\in \{1,2,\ldots ,s+1\}$
\end{rem}

In order to prove Theorem \ref{finalresult}, we need three lemmas, which we prove now.

\begin{lem}\label{1}
Let $A=(a_{ij})$ be an $n\times n$ real matrix such that $D(A)\in \mathcal{D}$. Let $D(A)$ be a strongly connected digraph and $\Delta _A \neq 0$. Let $B=(b_{ij})$ be the matrix given by $b_{ij}=\frac{\mu _{ij}}{\Delta _A}, 1\leq i,j\leq n$. Then, $AB=BA$.
\end{lem}
\begin{proof}
Let $A=(a_{ij})$. Then, $AB=BA$ has the following equivalent form :
\begin{equation}\label{ab=ba}
\sum _{k=1}^na_{ik}\mu _{kj}=\sum _{l=1}^n\mu _{il} a_{lj}\  \  \  for~every\  \  i,j\in \{1,2,\ldots n\}.
\end{equation}
First, we discuss the case $i=j$. Let $\{i_1,i_2,\ldots ,i_t\}\in \{1,2,\ldots ,n\}$ be such that for any $m\in \{1,2,\ldots ,t\}$, $a_{i,i_m}\neq 0$ and the cycle $(i,i_m,i)$ belongs to some maximum matching in $D(A)$. Since $\mathbb{M}(i,i_m)=\mathbb{M}(i_m,i)$ and $\beta _{\overline{i,i_m}}(M)=\beta _{\overline{i_m,i}}(M)$, the expressions on both the sides of equation (\ref{ab=ba}) are equal and they equal the common value $\sum_{m=1}^t\big (a_{ii_m}a_{i_mi}\sum _{M\in \mathbb{M}(i_m,i)}\beta _{\overline{i_m,i}}(M)\big)=\sum _{M\in \mathbb{M}(i)}\eta (M)$.

Now, we consider the fact $i\neq j$ and prove \ref{ab=ba} by considering four cases.\\
{\bf Case (i)}: $i$ and $j$ are pendant vertices.\\ 
\underline{{\it Subcase \rm{(i.1)}}}: $i$ and $j$ have a unique common neighbour. Let $q$ be the common neighbour. Then, $(i,q,i)$ and $(j,q,j)$ are both maximum matching cycles and cannot simultaneously be present in a maximum
matching. So, $\{M\backslash \{(i,q,i)\}\vert ~M\in \mathbb{M}(i,q)\}=\{M\backslash \{(j,q,j)\}\vert ~M\in \mathbb{M}(j,q)\}$. Thus,
\begin{align*}
a_{iq}\mu _{qj}& =a_{iq}a_{qj}\sum _{M\in \mathbb{M}(q,j)}\beta _{\overline{q,j}}(M) \\
 &= a_{qj}\Big (a_{iq}\sum _{M\in \mathbb{M}(i,q)}\beta _{\overline{i,q}}(M)\Big ) \\
 &= \mu _{iq}a_{qj}. 
\end{align*} 
It follows that equation \ref{ab=ba} holds.\\
\underline{{\it Subcase \rm{(i.2)}}}: There is no common neighbour for the vertices $i$ and $j$. Let vertices $i$ and $j$ be adjacent to non-pendant vertices $q$ and $p$, respectively. Then, the length of any cycle chain from $i$ to $j$ is at least three. As we have observed in Remark \ref{properties}, both the end vertices of a length three alternating cycle chain must be pendant vertices, $\mu _{qj}$ and $\mu _{ip}$ are both zero. So, equation \ref{ab=ba} is vacuously true.\\
{\bf Case (ii)}: $i$ and $j$ are non-pendant vertices. Since a non-pendant cycle can not be present in a maximum matching and the end vertices of a length three alternating cycle chain are pendant, $\mu _{qj}$ and $\mu _{ip}$ are zero for any arbitrary vertex $q$ adjacent to $i$ and $p$ adjacent to $j$ respectively. So, the expressions on both sides of equation \ref{ab=ba} are zero.\\
{\bf Case (iii)}: $i$ is a pendant vertex, while $j$ is a non-pendant vertex. Let $N_j=\{j_1,j_2,\ldots ,j_s\}$ be the set of all pendant vertices adjacent to $j$. We consider two subcases.\\
\underline{{\it Subcase \rm{(iii.1)}}}: $i$ is adjacent to $j$. then, the left hand side of \ref{ab=ba} is zero and the right hand side equal to
\begin{equation}\label{split}
\sum _{l=1}^s\mu _{ij_l} a_{j_lj} + \sum _{\{m\vert m \in N(j)\backslash N_j\}}\mu _{im} a_{mj}.
\end{equation}
The first sum is zero because $i$ and $j_l$ are pendant vertices with common neighbour $j$ for all $l\in \{1,2,\ldots ,s\}$. In the second sum, for any $m$, the vertex $m$ is non-pendant and also it is not adjacent to pendant vertex $i$. So, $\mu _{im}$ is zero.\\
\underline{{\it Subcase \rm{(iii.2)}}}: $i$ is not adjacent to $j$. Let $q$ be the non-pendant vertex adjacent to $i$. The analysis in this case is split into two further subcases, which we consider next. \\
\underline{{\it Subcase \rm{(iii.2.a)}}}: $q$ is adjacent to $j$. Since non-pendant cycles do not belong to any maximum matching, again, the left hand side of \ref{ab=ba} is zero, and the right hand side can be split as \ref{split}. Let $\mathbb{M}_{j_l}(i,q)\subseteq \mathbb{M}(i,q)$ be the set of maximum matching containing pendant cycle $(j,j_l,j)$. Then $\mathbb{M}(i,q)=\cup _{l=1}^s\mathbb{M}_{j_l}(i,q)$, a mutually disjoint union. Now, the second sum of \ref{split} is
\begin{align*}
\mu _{iq}a_{qj}& =a_{iq}a_{qj}\sum _{M\in \mathbb{M}(i,q)}\beta _{\overline{i,q}}(M) \\
 &= a_{qj}a_{iq}\sum _{l=1}^s\sum _{M\in \mathbb{M}_{j_l}(i,q)}\beta _{\overline{i,q}}(M) \\
 &= a_{qj}a_{iq}\sum _{l=1}^sa_{jj_l}a_{j_lj}\sum _{M\in \mathbb{M}_{j_l}(i,q)}\beta _{\overline{i,j_l}}(M) \\
 &= -\sum _{l=1}^sa_{j_lj}\big (-a_{iq}a_{qj}a_{jj_l}\big )\sum _{M\in \mathbb{M}(i,j_l)}\beta _{\overline{i,j_l}}(M) \\
 &= -\sum _{l=1}^sa_{j_lj}\mu _{ij_l}.
\end{align*} 
So, the right hand side of \ref{ab=ba} is also zero.\\
\underline{{\it Subcase \rm{(iii.2.b)}}}: $q$ is not adjacent to $j$. Since $q$ and $j$ both are non-pendant vertices, the left hand side of equation \ref{ab=ba} is zero. After splitting the right hand side as \ref{split}, in the first sum, the length of any cycle chain from vertex $i$ to $j_l$ must be at least four for any $l\in \{1,2,\ldots ,s\}$ and so, $\mu _{ij_j}$ is zero. Next, as $m$ is a non-pendant vertex and since it is not adjacent to $i$, the second sum is also zero, showing that the right hand side equals zero.\\
{\bf Case (iv)}: $i$ is a non-pendant vertex and $j$ is a pendant vertex. Let $N_i=\{i_1,i_2,\ldots ,i_r\}$ be the set of all pendant vertices adjacent to $i$. Then, the proof is the same as in Case (iii) by interchanging the roles of $i$ and $j$.
\end{proof}
In the next result, we present a graph theoretical interpretation for the entries of the product $AB$, where $A$ and $B$ are as defined in Lemma \ref{1}.

\begin{cor}\label{modified first equation}
Let $A$ and $B$ satisfy the hypotheses of Lemma \ref{1}. Then,\\
$$(AB) _{ii}=\left\{\begin{array}{cl} 1 &\mbox{if }i\mbox{ is a non-pendant vertex, }\\ \frac{\sum_{M\in \mathbb{M}(i)}\eta (M)}{\Delta _A} &\mbox{if }i\mbox{ is a pendant vertex}\end{array}\right.$$
while for $i\neq j$,
$$(AB) _{ij}=\left\{\begin{array}{cl} \frac{a_{qj}\mu _{iq}}{\Delta _A} &\mbox{if }i,j\mbox{ are pendant vertices and }\\  & \mbox{have a common neighbour } q,\\ 0 &\mbox{otherwise. }\end{array}\right.$$
\end{cor}
\begin{proof}
It follows from Lemma \ref{1} that $(AB)_{ii}=\frac{\sum_{M\in \mathbb{M}(i)}\eta (M)}{\Delta _A}$. By Remark \ref{properties}, a non-pendant vertex is matched in every maximum matching, and so for a non-pendant vertex $i$, $(AB)_{ii}=\frac{1}{\Delta _A}\cdot \Delta _A=1$. For $i\neq j$, the proof is followed by observing case by case in the above Lemma.
\end{proof}
\begin{rem}\label{zero}
In the above corollary, for $i\neq j$, $(AB)_{ij}$ is not equal to zero only when $i,j$ are pendant vertices and have common neighbour.
\end{rem}
\begin{lem}\label{2}
Let $A$ and $B$ satisfy the hypotheses of Lemma \ref{1}. Then, $ABA = A$.
\end{lem}
\begin{proof}
To show that $ABA=A$, we show
\begin{equation}\label{ABA=A}
\sum _{k=1}^n(AB)_{ik}a_{kj}=\left\{\begin{array}{ll}
a_{ij}, & \mbox{when }(i, j) \mbox{ is an edge} \\
0, & \mbox{when } (i, j)\mbox{ is not an edge}.
\end{array}\right .
\end{equation}
Let $c$ be the left hand side of (\ref{ABA=A}). Then, $c$ can be written in the form $c=c_i+\tilde{c_i}$, where
$$c_i=(AB)_{ii}a_{ij}\  \  \  \mbox{and}\  \  \  \  \tilde{c_i}=\sum _{\substack {k=1 \\ k\neq i}}^n(AB)_{ik}a_{kj}.$$
Then, the rest of the calculations can be done by giving the same arguments as in the proof of \cite[Theorem 2.8]{nandi}.
\end{proof}
\begin{lem}\label{3}
Let $A$ and $B$ satisfy the hypotheses of Lemma \ref{1}. Then, $BAB = B$.
\end{lem}
\begin{proof}
By Lemma \ref{1}, if we prove $ABB=B$, then we are done. This is equivalent to proving that,
\begin{equation}\label{bab=b}
\sum _{k=1}^n(AB)_{ik}\mu _{kj}=\left\{\begin{array}{ll}
\mu _{ij}, & \mbox{if }i,j\mbox{ are maximally matchable, }  \\
0, & \mbox{otherwise}.
\end{array}\right .
\end{equation}
Let $b$ be the left hand side of (\ref{bab=b}). Then, $b$ can be written in the form $b=b_i+\tilde{b_i}$, where
$$b_i=(AB)_{ii}\mu _{ij}\  \  \  \mbox{and}\  \  \  \  \tilde{b_i}=\sum _{\substack {k=1 \\ k\neq i}}^n(AB)_{ik}\mu _{kj}.$$
Then, again the rest of the proof same as the proof of \cite[Theorem 2.7]{nandi}.
\end{proof}
\begin{proof}[Proof of Theorem \ref{finalresult}]
It follows from Lemma \ref{1}, \ref{2} and \ref{3}.
\end{proof}
Here is an illustration.
\begin{ex}
Consider the matrix
$$A=\begin{pmatrix}
~0 & -2 & ~0 & ~2 & 2 & 1 & ~0 & 0 & 0 & 0\\
-1 & ~0 & ~1 & ~0 & 0 & 0 & -3 & 0 & 0 & 0\\
~0 & ~3 & ~0 & -1 & 0 & 0 & ~0 & 2 & 2 & 0\\
-1 & ~0 & ~1 & ~0 & 0 & 0 & ~0 & 0 & 0 & 2\\
~1 & ~0 & ~0 & ~0 & 0 & 0 & ~0 & 0 & 0 & 0\\
-3 & ~0 & ~0 & ~0 & 0 & 0 & ~0 & 0 & 0 & 0 \\
~0 & -2 & ~0 & ~0 & 0 & 0 & ~0 & 0 & 0 & 0\\
~0 & ~0 & ~3 & ~0 & 0 & 0 & ~0 & 0 & 0 & 0\\
~0 & ~0 & -1 & ~0 & 0 & 0 & ~0 & 0 & 0 & 0\\
~0 & ~0 & ~0 & ~2 & 0 & 0 & ~0 & 0 & 0 & 0\\
\end{pmatrix}.$$
Then, $D(A)$ is the digraph $D_1$ in Figure \ref{D1}. Here, all the maximum matchings of $D(A)$ are given by: 
$$M_1=\{(1,5,1),(2,7,2),(3,8,3),(4,10,4)\},$$
$$M_2=\{(1,5,1),(2,7,2),(3,9,3),(4,10,4)\},$$
$$M_3=\{(1,6,1),(2,7,2),(3,8,3),(4,10,4)\}$$
and
$$M_4=\{(1,6,1),(2,7,2),(3,9,3),(4,10,4)\}.$$
So, $\Delta _A=288+(-96)+(-432)+144=-96$. Let $A^{\#}=(\alpha _{ij})$. Let us compute $\alpha _{57}$. First, $P_3(5,7)=1\times (-2)\times (-3)=6$. Note that $C_3(5,7)$ cycle chain is alternating with respect to the maximum matchings $M_1$ and $M_2$. Thus, $\beta _{57}=(-1)\times 6=-6$, $\beta _{\overline{5,7}}(M_1)=6\times 4=24$ and $\beta _{\overline{5,7}}(M_2)=(-2)\times 4=-8$. So,
$$\mu _{57}=(-6)\times (24-8)=-96.$$
Therefore, $\alpha _{57}=\frac{-96}{-96}=1$.
\end{ex}
\section{Necessary and sufficient condition of a real matrix $A$ such that $D(A),D(A^{\#})\in \mathcal{D}$}\label{section3}
Before we proceed, let us provide an example. 
\begin{ex}\label{twoexample}
First we consider a matrix
$$A=\begin{pmatrix}
0 & ~2 & -1 & 0 & 0\\
1 & ~0 & ~0 & 1 & 1\\
1 & ~0 & ~0 & 0 & 0\\
0 & -2 & ~0 & 0 & 0\\
0 & -2 & ~0 & 0 & 0\\
\end{pmatrix}.$$
Then $D(A)\in \mathcal{D}$ and
$$A^{\#}=\begin{pmatrix}
~0 & 0 & ~1 & ~0 & ~0\\
~0 & 0 & ~0 & -\frac{1}{4} & -\frac{1}{4}\\
-1 & 0 & ~0 & -\frac{1}{2} & -\frac{1}{2}\\
~0 & \frac{1}{2} & -\frac{1}{2} & ~0 & ~0\\
~0 & \frac{1}{2} & -\frac{1}{2} & ~0 & ~0\\
\end{pmatrix}.$$
Since the non-pendant vertices 2, 4 and 5 are not adjacent to any pendant vertex in $D(A^{\#})$, $D(A^{\#})\notin \mathcal{D}$. On the other hand, for a matrix $$B=\begin{pmatrix}
~0 & 1 & 1 & 2 & -1\\
-1 & 0 & 0 & 0 & ~0\\
~2 & 0 & 0 & 0 & ~0\\
~1 & 0 & 0 & 0 & ~0\\
~2 & 0 & 0 & 0 & ~0\\
\end{pmatrix},$$
$D(B)$ is a star tree and belongs to $\mathcal{D}$. Since $B^{\#}=B$, $D(B^{\#})$ also belongs to $\mathcal{D}$.
\end{ex}
The two examples above serve a good motivation for our main result (Theorem \ref{maintheorem}) in this section.

The first result in this section shows that the $ij$-th entry of the group inverse of a real matrix $A$ with $D(A)\in \mathcal{D}$ is non-zero when $i,j$ are maximally matchable in $D(A)$.
\begin{cor}\label{nonzero}
Let $A$ satisfy the hypothesis of Theorem \ref{finalresult}. If $i,j$ are maximally matchable and $A^{\#}=(\alpha _{ij})$, then $\alpha _{ij}\neq 0$.
\end{cor}
\begin{proof}
The proof follows from the proof of \cite[Corollary 2.12]{nandi}.
\end{proof}
In general, it is not true that the digraph corresponding to the group inverse of a matrix with a simple symmetric digraph is again a simple symmetric digraph, shown by the following example.
\begin{ex}\label{ssd}
Consider a matrix with simple symmetric digraph, 
\begin{center}
$A=\begin{pmatrix}
0 & 2 & 1 & 2 & 1\\
2 & 0 & 2 & 0 & 0\\
1 & 2 & 0 & 0 & 0\\
2 & 0 & 0 & 0 & 0\\
1 & 0 & 0 & 0 & 0\\
\end{pmatrix}$. Then, its group inverse
$A^{\#}=\begin{pmatrix}
0 & \ \ 0 & \ \ 0 & \ \ \frac{2}{5} & \ \ \frac{1}{5}\\
0 & \ \ 0 & \ \ \frac{1}{2} & -\frac{1}{5} & -\frac{1}{10}\\
0 & \ \ \frac{1}{2} & \ \ 0 & -\frac{2}{5} & -\frac{1}{5}\\
\frac{2}{5} & -\frac{1}{5} & -\frac{2}{5} & \ \ \frac{8}{25} & \ \ \frac{4}{25}\\
\frac{1}{5} & -\frac{1}{10} & -\frac{1}{5} & \ \ \frac{4}{25} & \ \ \frac{2}{25}\\
\end{pmatrix}$.
\end{center}
It is clear that $D(A)$ is a simple symmetric digraph, while $D(A^{\#})$ is not a simple symmetric digraph.
\end{ex}
In the following result, we show that when $D(A)\in \mathcal{D}$ for a real matrix $A$, $D(A^{\#})$ is a simple symmetric digraph.
\begin{cor}\label{ssd2}
Let $A$ be an $n\times n$ real matrix such that $D(A)\in \mathcal{D}$. Assume that $D(A)$ is strongly connected and $\Delta _A\neq 0$. Then, $D(A^{\#})$ is a strongly connected simple symmetric digraph.
\end{cor}
\begin{proof}
Since $D(A)$ is strongly connected, $A$ is an irreducible matrix. Then by \cite[Lemma 2.4]{kls}, $A^{\#}$ is irreducible. So, $D(A^{\#})$ is strongly connected. Let $A^{\#}=(\alpha _{ij})$. Then $\mu _{ii}=0$ implies $\alpha _{ii}=0$ and so, $D(A^{\#})$ is simple. For any $i$ and $j$, suppose $\alpha _{ij}\neq 0$. Then, by Theorem \ref{finalresult}, $i,j$ are maximally matchable. Clearly, $j,i$ are also maximally matchable and so, from Corollary \ref{nonzero}, $\alpha _{ji}\neq 0$. Thus, $A^{\#}$ is a combinatorially symmetric matrix and $D(A^{\#})$ is a simple symmetric digraph.
\end{proof}
Recall that a simple symmetric digraph is said to be {\it corona digraph} if each non-pendant vertex is adjacent to exactly one pendant vertex.
\begin{prop}\label{starcorona}
Let $A=(a_{ij})$ be an $n\times n$ real matrix such that $\Delta _A\neq 0$.\\
\rm{$(i)$} If $D(A)$ is a star tree digraph, then $D(A^{\#})$ is a star tree digraph.\\
\rm{$(ii)$} If $D(A)$ is a corona digraph, then $D(A^{\#})$ is a corona digraph.
\end{prop}
\begin{proof}
$(i)$ Let the vertex set of $D(A)$ be $\{1,2,\ldots ,n\}$. Without loss of generality, let the center vertex of $D(A)$ be $n$. Then, the only non zero entries in $A$ are $a_{in},a_{ni}$ for all $1\leq i\leq n-1$. On the other hand, the maximum matchings of $D(A)$ are $M_i=\{(n,i,n)\}$ and so, the only maximum matchable pair of vertices are $n,i$ for all $1\leq i\leq n-1$. Let $A^{\#}=(\alpha _{ij})$. Then, by Theorem \ref{finalresult} and Corollary \ref{nonzero}, the only non zero entries of $A^{\#}$ are $\alpha _{in},\alpha _{ni}$ for all $1\leq i\leq n-1$. Thus, $D(A^{\#})$ is a star tree digraph.

$(ii)$ If $D(A)$ is a corona digraph, then $n$ is even. Let the vertex set of $D(A)$ be $\{1,2,\ldots ,\frac{n}{2}, \frac{n}{2} +1,\ldots n\}$. Without loss of generality, let the set of non-pendant vertices and pendant vertices of $D(A)$ be $\{1,2,\ldots ,\frac{n}{2}\}$ and $\{\frac{n}{2} +1,\frac{n}{2} +2,\ldots n\}$, respectively, where vertex $k$ is adjacent to vertex $\frac{n}{2}+k$ for all $1\leq k\leq \frac{n}{2}$. Then, the matrix $A$ is of the form \ref{glsform}, where $F$ and $G$ both are diagonal matrices of order $\frac{n}{2}$ with non-zero entries in the diagonal. Also, $A$ is invertible and 
\begin{center}
$A^{\#}=A^{-1}=\begin{pmatrix}
0 & G^{-1} \\
F^{-1} & -F^{-1}EG^{-1} \\
\end{pmatrix}.$
\end{center}
Since $F$ and $G$ are both diagonal matrices with non-zero entries in the diagonal, the zero non-zero pattern of $E$ and $-F^{-1}EG^{-1}$ will be the same. Then, there exists a permutation matrix $P$ such that
\begin{center}
$P^{-1}A^{\#}P=\begin{pmatrix}
-F^{-1}EG^{-1} & F^{-1} \\
G^{-1} & 0 \\
\end{pmatrix}.$
\end{center}
Since the zero non-zero pattern of $A$ and $P^{-1}A^{\#}P$ are same, $D(A^{\#})$ is also a corona digraph.
\end{proof}
Finally, we classify all the real square matrix $A$ such that $D(A)$ and $D(A^{\#})$ both are in the class $\mathcal{D}$, that is $D(A)$ is either a corona digraph or a star digraph.
\begin{proof}[Proof of Theorem \ref{maintheorem}]
Let $D(A)$ be a corona or star tree digraph. Then, by Proposition \ref{starcorona}, $D(A^{\#})\in \mathcal{D}$. Conversely, suppose $D(A)$ is neither a corona nor a star tree digraph. Since $D(A)\in \mathcal{D}$, there exists a non-pendant vertex adjacent to more than one pendant vertices and at least two non-pendant vertices in $D(A)$. Let $i$ be a non-pendant vertex adjacent to $s\geq 2$ pendant vertices $\{i_1,i_2,\ldots ,i_s\}$ and $i$ is adjacent to another non-pendant vertex $j$. The alternating cycle chains that start or end with vertex $i$ are $((i,i_m,i))$ for all $1\leq m\leq s$. So, $i$ is adjacent to exactly $s$ vertices $\{i_1,i_2,\ldots ,i_s\}$ in $D(A^{\#})$. Let $j$ be adjacent to a pendant vertex $\tilde{j}$. Then, for any $1\leq m\leq s$, $((i_m,i,i_m))$ and $((i_m,i,i_m),(i,j,i),(j,\tilde{j},j))$ are two alternating cycle chains start from vertex $i_m$. So, vertex $i_m$ is non-pendant in $D(A^{\#})$ for all $1\leq m\leq s$. Since we have a non-pendant vertex $i$ which is not adjacent to any pendant vertices in $D(A^{\#})$, $D(A^{\#})$ does not belong to $\mathcal{D}$.
\end{proof}

\section*{Acknowledgments}
The author thanks the Indian Institute of Information Technology, Design \& Manufacturing, Kurnool for providing a good research environment.

\section*{Declarations}
\noindent{\textbf{Conflict of Interest:}} The authors declare that there is no conflict of interest with anyone.

\end{document}